\documentclass{amsart}
\usepackage[utf8]{inputenc}
\usepackage{amssymb, amsmath,amsfonts,amsthm,amstext}
\usepackage{calrsfs}
\usepackage{hyperref}
\usepackage{mathtools}
\usepackage{tikz-cd}

\newtheorem{theorem}{Theorem}[section]
\newtheorem{lemma}[theorem]{Lemma}
\newtheorem{prop}[theorem]{Proposition}
\newtheorem{thm}{Theorem}
\newtheorem{coro}[thm]{Corollary}
\newtheorem{corollary}[theorem]{Corollary}

\theoremstyle{definition}

\newtheorem{example}[theorem]{Example}
\newtheorem*{notation*}{Notation}

\newcommand{\cd}{\operatorname{cd}}
\newcommand{\charac}{\operatorname{char}}

\newcommand{\len}{\operatorname{len}}

\newcommand{\mf}[1]{\mathfrak{#1}}

\newcommand{\llangle}{{\langle \langle}}
\newcommand{\rrangle}{{\rangle \rangle}}
\newcommand{\N}{\mathbb{N}}

\newcommand{\g}{\mathfrak{g}}

 \hypersetup{
   colorlinks=true,
     linkcolor=blue,
    }

\title{Some applications of Gröbner-Shirshov bases to Lie algebras}
\author{Luis Mendonça}
\address{Instituto de Ciências Exatas, Universidade Federal de Minas Gerais (UFMG), 31270-901 Belo Horizonte, Brazil.}
\email{luismendonca@mat.ufmg.br}
\begin{document}

 \begin{abstract}
 We show that if a countably generated Lie algebra $H$ does not contain isomorphic copies of certain finite-dimensional nilpotent Lie algebras $A$ and $B$ (satisfying some mild conditions), then $H$ embeds into a quotient of $A \ast B$ that is at the same time hopfian and cohopfian. This is a Lie algebraic version of an embedding theorem proved by C. Miller and P. Schupp for groups. We also prove that any finitely presentable Lie algebra is the quotient of a finitely presented, centerless, residually nilpotent and SQ-universal Lie algebra of cohomological dimension at most $2$ by an ideal that can be generated by two elements as a Lie subalgebra. This is reminiscent of the Rips construction in group theory. In both results we use the theory of Gröbner-Shirshov bases.
 \end{abstract}

\maketitle

\section{Introduction}
Throughout this note, we consider Lie algebras over an arbitrary field $K$.

A Lie algebra $L$ is said to be \emph{cohopfian} if it is not isomorphic to a proper subalgebra. Equivalently, any injective homomorphism $\varphi \colon L \to L$ is already an isomorphism. Dually, $L$ is \emph{hopfian} if any surjective homomorphism $\psi \colon L \to L$ is an isomorphism.

Finite-dimensional Lie algebras are evidently both hopfian and cohopfian. It is a classical fact that finitely generated free Lie algebras are hopfian, but clearly the non-abelian ones are not cohopfian. Any simple Lie algebra is obviously hopfian. Non-hopfian Lie algebras are rarer, but one can even find a finitely presented one such as
\[ \Gamma = \langle x,y,z,t \mid [x,[y,z]]=z, \ \ [x,t] = 0, \ \ [y,t]=0, \ \ [y,z] = [z,t] \rangle;\]
see \cite[page~117]{BokutKukinBook}. Further information and examples of (non-)(co)hopfian algebras may be found in \cite{RowenSmall2017, PassmanSmall2021}.

We denote by $F(L)$ the Frattini subalgebra of a Lie algebra $L$, that is, its subalgebra of non-generator elements. Moreover, $A \ast B$ denotes the free product of two Lie algebras $A$ and $B$. We will prove the following embedding result.

\begin{thm}  \label{thm.embedding}
 Let $H$ be a countably generated Lie algebra and let $A$ and $B$ be two finite-dimensional Lie algebras that do not embed into $H$. Suppose further that $\dim A = \dim B \geq 4$, that $F(A) \neq 0$ and that $\dim F(B) \geq 2$. Then
 \begin{enumerate}
  \item\label{part1thm1} $H$ can be embedded in a cohopfian quotient $L$ of $A \ast B$,
  \item\label{part2thm1} if $A$ and $B$ are nilpotent but not isomorphic to each other, then $H$ can be embedded in a quotient $L$ of $A \ast B$ that is both hopfian and cohopfian.
 \end{enumerate}
\end{thm}

The motivation for this investigation is a theorem of C. Miller and P. Schupp in group theory. In \cite{MillerSchupp1971}, the authors show that every countable group $G$ is embeddable in a hopfian quotient $Q$ of the free product $C_p \ast C_s$ of cyclic groups, provided that $p\geq 2$ and $s \geq 3$. Furthermore, if $G$ does not contain elements of order $p$ or $s$, then $Q$ can be chosen to be cohopfian as well.

The conditions that we require on $A$ and $B$ are not strong obstacles to the applicability of part~\ref*{part1thm1} of the result. Indeed, it is enough to replace the finite-dimensional Lie algebra that is not embeddable in $H$ with a direct sum of it with a suitable nilpotent Lie algebra. We believe however that the conditions could be relaxed at the cost of proving some technical lemmas similar to Lemma~\ref{L1}.

The tool used in \cite{MillerSchupp1971} is the small cancellation theory of presentations of groups. Part of the usefulness of such theory lies in the fact that, if in a group presentation the defining relators have small overlap compared to their lengths, then one can determine whether a given word in the generators represents the trivial element. In this paper we make we use of the concept of Gröbner-Shirshov bases. Such theory was proposed by Shirshov \cite{Shirshov2} and developed further by Bokut and other authors (see e.g. \cite{BokutKukinBook}). It allows us to define Lie algebras by generators and relators in a way that we can again decide whether a given element is trivial or not. We review in Section~\ref{prelim.GS} the basic facts of the theory.

Another important and influential result in group theory is a theorem of E. Rips \cite{Rips1982}. There it is proved that any finitely presented group $Q$ is a quotient of a certain small cancellation group $G$ by a normal subgroup $N$ that can be generated by $2$ elements (as a subgroup). Such construction is quite useful to build examples of subgroups of small cancellation groups having some unexpected properties, especially from the point of view of decision problems.

More recently, new versions of such construction, where $G$ is required to satisfy some extra conditions,  were studied by other authors. For instance, a version where $G$ is residually finite but $N$ requires $3$ generators was obtained by D. Wise \cite{Wise2003}.

Using Gröbner-Shirshov bases again, we prove the following result.

\begin{thm} \label{RipsVersion}
 Let $Q$ be a finitely presented Lie algebra. Then there exists an exact sequence
\[0 \to N \to L \to Q \to 0\]
where $L$ is finitely presented, centerless, of cohomological dimension at most $2$, residually nilpotent and SQ-universal, and $N$ can be generated (as a Lie algebra) by two elements.
\end{thm}

Saying that $L$ is SQ-universal means that any countably generated Lie algebra embeds in a quotient of $L$. In the above, $L$ is defined by a finite Gröbner-Shirshov basis without overlap ambiguity (see Section~\ref{prelim.GS}). So Theorem~\ref{RipsVersion} can be interpreted as a result on badness of subalgebras of such Lie algebras, as much as the Rips construction is about badness of subgroups of small cancellation groups.

By choosing $Q$ to be a finitely presentable Lie algebra with unsolvable word problem, the Lie algebra $L$ determined in the theorem has unsolvable \emph{generalized word problem:} one cannot decide whether a given element of $L$, written in terms of its generators, is an element of the finitely generated subalgebra $N$.

Another possible application of Theorem~\ref{RipsVersion} is available when we choose $Q$ to be perfect (i.e. $Q = [Q,Q]$) and satisfying $H_2(Q,K)=0$, as was done by M. Bridson and A. Reid \cite{BridsonReid2014} for groups. Then it happens that the map $N \to L$ induces isomorphisms on the lower central quotients.

\begin{coro}  \label{coro.paraequiv}
 There exist finitely generated, residually nilpotent Lie algebras $A$ and $B$ and a homomorphism $A \to B$ that induces  isomorphisms $A/A^n \simeq B/B^n$ for all $n$ where $B$ is finitely presented, but $A$ is not.
\end{coro}

We construct examples where $A$ can be chosen to have infinite-dimensional second homology group $H_2(A,K)$.

We observe that S.O. Ivanov, R. Mikhailov and A. Zaikovskii recently produced countably generated \emph{parafree} Lie algebras $L$ with $H_2(L,K) \neq 0$ \cite{IMZ2020}, where $K  $ is a field of characteristic $2$. Being parafree means that $L$ is residually nilpotent and there is a homomorphism $F \to L$, where $F$ is a free Lie algebra, that induces isomorphisms $F/F^n \to L/L^n$ for all $n$. So the setting is similar to the above: the authors produce Lie algebras that have indistinguishable nilpotent quotients, but differ on some other aspect.

\begin{notation*}
 For a subset $X$ of a Lie algebra $L$, we will denote by $\langle X \rangle$ and by $\llangle X \rrangle$ the subalgebra and the ideal, respectively, generated by $X$ in $L$.
\end{notation*}

\section{Gröbner-Shirshov bases}  \label{prelim.GS}
We recall here the basic facts of the theory of Gröbner-Shirshov bases for Lie algebras. For more details, we refer the reader to \cite{BokutChenMo2010, BokutKukinBook, Shirshov1, Shirshov2}.

Let $F$ be the free Lie algebra with a well-ordered free basis $X = \{x_1, x_2, \ldots\}$. For the elements of the free monoid $X^{\ast}$, we consider a lexicographic ordering such that $x_1 <_{lex} x_2 <_{lex} \cdots$ and $v <_{lex} u$ if $u$ is a proper prefix of $v$. A word $w\in X^{\ast}$ is said to be \textit{regular} if $w=uv$ implies $uv>_{lex} vu$ for  non-trivial subwords $u,v$, that is, $w$ is lexicographically greater than all of its cyclic permutations.

Let $[w]$ denote some bracketing of the word $w$, that is, $[w]$ is a non-associative word. We say that $[w]$ is \textit{non-associative regular} if $w$ is regular and:
\begin{enumerate}
 \item If $[w]=[u][v]$, then both $[u]$ and $[v]$ are non-associative regular, and
 \item If $[w]=[ [u_1][u_2] ][v]$, then $u_2 \leq_{lex} v$.
\end{enumerate}
Any regular word $w$ admits a unique bracketing $(w)$ that is non-associative regular, and the set of all non-associative regular words on some alphabet $X$ is a basis of the free Lie algebra generated by $X$ \cite{Shirshov1}. An inductive way to find the regular bracketing of a regular word goes as follows: if $w=uv$, where $v$ is the longest proper regular suffix of $w$, then $(w) = [ (u), (v) ]$. For instance if $a>b>c$ are elements of $X$, then the regular bracketing of $w = a^2cb$ is $(w) = [ a, [[a,c],b]]$, because $acb$, $b$ and $c$ are the longest proper regular suffixes of $w$, $acb$ and $ac$, respectively.

\begin{example}  \label{words.wi}
If $a, b \in X$ and $a>b$, then the words $w_i = a^2b^i a b$ are regular, and their regular bracketing is
\[ (w_i) = [ [a, [a, \underbrace{ b, \ldots, b}_{i \text{ times}}]] ,   [a,b] ].\]
These words were used in \cite{BokutChenMo2010} as a tool in the proof of some embedding theorems, and we will use them here too.
\end{example}

In view of the above, we may write elements of $F$ as linear combinations of regular words without bracketings, and it should be understood that they are endowed with the regular bracketing. For instance, $f = a^2 c b - b d$ is the element $[ a, [[a,c],b]]-[b,d]$ when $a>b>c>d$.

We will also need to consider the deg-lex ordering of words. Let us write $\len(w)$ for the length of a word. We put $u< v$ for $u,v \in X^{\ast}$ if either $\len(u) <\len(v)$, or if the lengths are the same but $u <_{lex} v$. This translates into an ordering of the basis of $F$: by definition $(u)< (v) \Leftrightarrow u < v$.

Any element $f \in F$ is uniquely written as a linear combination of regular non-associative words. The leading word of $f$, denoted by $\bar{f}$, is the highest (with respect to the deg-lex ordering) corresponding regular associative word appearing with non-zero coefficient in the expression for $f$. We say that $f$ is \textit{monic} if the coefficient of the term associated with $\bar{f}$ in $f$ is $1 \in K$.

Let $f, g \in F$. The words $\bar{f}$ and $\bar{g}$ are associative regular, and so is $w = \bar{f} \bar{g}$ if $\bar{f}>_{lex} \bar{g}$, but it may be the case that the regular bracketing of $w$ is not simply $[(\bar{f}), (\bar{g})]$. On the other hand, the following holds.

\begin{lemma} \label{LT.product}
 Let $f,g \in F$ with $\bar{f} >_{lex} \bar{g}$. Then the leading word of $[f,g]$ is $\bar{f}\bar{g}$.
\end{lemma}
\begin{proof}
 See \cite[Corollary~2.11.14]{BokutKukinBook}.
 \end{proof}

In order to define the concept of Gröbner-Shirshov bases we need to consider two types of composition of elements. Suppose that $f,g \in F$ are monic elements with $\bar{f}= ab >_{lex} bc = \bar{g}$ for some non-trivial words $a$, $b$ and $c$. In the language of \cite{Bergman1978}, the words $\bar{f}$ and $\bar{g}$ have \emph{overlap ambiguity}. By Shirshov's Lemma \cite[Lemma~2.11.15]{BokutKukinBook}, there are two special bracketings $[u]_1$ and $[u]_2$ for the regular word $u=abc$ such that $[u]_1$ extends the regular bracketing $(\bar{f})$ of $\bar{f}=ab$ and $[u]_2$ extends the regular bracketing $(\bar{g})$ of $\bar{g}=bc$, and $\overline{ [u]}_1  = \overline{[u]}_2=u$.

Let $u_1$ and $u_2$ be the elements of $F$ obtained from $[u]_1$ and $[u]_2$ by substituting $(\bar{f})$ and $(\bar{g})$ with $f$ and $g$, respectively. The \textit{first-order composition} $(f,g)^I_u$ of $f$ and $g$ with respect to $u$ defined as
\[(f,g)^I_u = u_1 - u_2.\]
The important fact here is that the highest terms of $u_1$ and $u_2$ cancel out in the computation of $(f,g)^I_u$.

Similarly, suppose that $\bar{g}$ is a subword of $\bar{f}$ for some monic elements $f,g \in F$. Again by \cite[Lemma~2.11.15]{BokutKukinBook} we can find a bracketing $[u]_0$ of $u = \bar{f}$ that extends the regular bracketing $(\bar{g})$ of $\bar{g}$, and $\overline{[u]}_0 = u$. By substituting $(\bar{g})$ with $g$ in $[u]_0$ we obtain an element $f_{\ast}$ of $F$.  The \textit{second-order composition} $(f,g)^{II}_u$ of $f$ and $g$ is
\[(f,g)^{II}_u = f - f_{\ast}.\]

A subset $G \subset F$ is \textit{reduced} if all its elements are monic and if there are no distinct elements $f, g \in G$ such that $\bar{f}$ is a subword of $\bar{g}$. In other words, no second-order composition is possible between elements of $G$.

Finally, a \textit{Gröbner-Shirshov basis} (GS-basis, for short) of an ideal $I \subseteq F$ is a reduced set $G \subset F$ that generates $I$ as an ideal and such that if $f, g \in G$ define a first-order composition with respect to some word $u$, then
\[ (f,g)^I_u = \sum_{i=1}^m h_i,\]
where each $h_i$ lies in $\langle \langle r_i \rangle \rangle$ for some $r_i \in G$ and $\bar{h}_i < u$ for $1 \leq i \leq m$. In this case, we say that the overlap ambiguity of $f$ and $g$ is \emph{solvable}.

\begin{theorem}\cite[Lemma~3.2.7]{BokutKukinBook} \label{thmShirshov}
 Let $G$ be a Gröbner-Shirshov basis of the ideal $I \subseteq F$ and let $0 \neq f \in F $. If $f \in I$, then
 $\bar{f}$ contains $\bar{r}$ as a subword for some $r \in G$.
\end{theorem}

For a GS-basis $G$, we say that a word $w$ is \emph{$G$-irreducible} if it does not contain any $\bar{r}$, with $r \in G$, as a subword. So the theorem says that $f \neq 0$ in $F/I$ if $\bar{f}$ is $G$-irreducible. Another formulation is the following.

\begin{corollary}  \cite[Lemma~3.2.8]{BokutKukinBook}.
 If $G \subset I \subseteq F$ is a Gröbner-Shirshov basis, then the set of the images of all the $G$-irreducible words forms a basis of $F/I$ as a vector space.
\end{corollary}

The process of writing an arbitrary element of $F/I$ as a linear combination of basis elements is as follows. If $g = \sum_i \mu_i g_i \in F$, but some $\bar{g}_j$ contains $\bar{r}$ as a subword, for some $r \in G$, we change $g$ with $g' = g - \mu_j p$, where $p$ is an element of the ideal generated by $r$ and and such that $\bar{p} = \bar{g}_j$, such as in the definition of second-order composition. Notice that $g$ and $g'$ are equivalent in $F/I$. By applying such transformations $g \to g'$ finitely many times, we reach a linear combination of $G$-irreducible words.

We note that any ideal $I$ of a free Lie algebra $F$ can be generated by a suitable (possibly infinite) Gröbner-Shirshov basis; see \cite[Lemma~3.2.5]{BokutKukinBook}.
\begin{example}
 If $L$ is a Lie algebra with basis $\{e_i\}_{i \in M \subseteq \N}$ and operation determined by $[e_i,e_j] =\sum_{t \in M} \theta_{i,j}^t e_t$, then $G = \{ e_i e_j - \sum_{t \in M} \theta_{i,j}^t e_t \mid i > j\}$ is a GS-basis in the free Lie algebra $F$ with free basis $\{e_i\}_{i \in M}$.
\end{example}

\begin{example}  \label{ex.wi}
 Let $G = \{w_i\}_{S \subseteq \N}$, where $w_i$ are the words from Example~\ref{words.wi} and $S \subseteq \N$ is any subset. There is no overlap between elements of $G$, so $G$ is a GS-basis for the ideal that it generates in the free Lie algebra containing generators  $a>b$.
\end{example}

Another important fact about GS-bases of Lie algebras is that they behave well with respect to the passage to the universal enveloping algebra, in the sense that if $G$ is a GS-basis in the free Lie algebra $F$, then its image in $U(F)$ is a GS-basis in the sense of associative algebras (essentially, its image is suitable to the application of some version of the Diamond Lemma such as in \cite{Bergman1978}). In particular, we have the following.
\begin{prop}  \label{cd1}
 Suppose that $G \subset F$ is a Gröbner-Shirshov basis without overlap between its elements and let $I$ be the ideal of $F$ generated by $G$. Then the Lie algebra $L = F/I$ is of cohomological dimension at most $2$.
\end{prop}
\begin{proof}
 The image of $G$ in $U(F)$ is a GS-basis without overlap in the sense of associative algebras, so Anick's resolution \cite{Anick1986} for $K$ over $U(L)$ is of length $2$.
\end{proof}

The next example will be used in Section~\ref{Applications}.
\begin{example}   \label{exV}
 Set $x>y>z$ and let
 \[ V = \langle x, y, z \mid x^2 y xy - x,\ \ x^2 y^2 xy - y, \ \ x^2y^3xy-z\rangle.\]
 It is clear that $V$ is perfect. The defining relators form a GS-basis $G$ without overlap, so $V$ has cohomological dimension $2$. In fact, tensoring Anick's resolution with $K$ over $U(V)$, we obtain  $0 \to K^3 \to K^3 \to K \to 0$, because $V$ has three generators and three relators. Using that $H_1(V,K) = 0$, we conclude that $H_2(V,K)=0$ too. Another piece of information that we will use is that the Lie subalgebra $\langle y, z \rangle$ is free, since every word depending only on $y$ and $z$ is $G$-irreducible.
\end{example}

\section{An analogue of the Rips construction}
In this section we prove Theorem~\ref{RipsVersion}. Let
\[Q= \langle x_1, \ldots, x_d \mid r_1, \ldots, r_m \rangle\]
be a finitely presented Lie algebra. Write each $r_i$ as a linear combination of regular words on the generators $x_1 < \cdots < x_d$. Let $q\geq 2d$ be an integer greater than the length of the longest word appearing on any $r_i$. Define $L$ as the Lie algebra generated by $x_1 < \cdots < x_d < b < a$, with the set $G$ of defining relators given by:
\begin{enumerate}
 \item\label{type1} $a^2 b^i a b - ax_i$, for $1 \leq i \leq d$,
 \item\label{type2} $a^2 b^{d+i} a b - bx_i$, for $1 \leq i \leq d$, and
 \item\label{type3} $a^2 b^{q+i} a b - r_i$, for $1 \leq i \leq m$.
\end{enumerate}
Notice that $G$ is a finite GS-basis: the leading terms of its elements are as in Example~\ref{ex.wi}, so there is no overlap. In particular, $L$ has cohomological dimension at most $2$ by Proposition~\ref{cd1}.

The obvious map $L \to Q$ that kills $a$ and $b$ is clearly surjective. Moreover, its kernel is generated by $a$ and $b$ as a Lie subalgebra, in view of the relators of types~\ref*{type1} and~\ref*{type2}.

\begin{prop}
 $L$ is SQ-universal.
\end{prop}
\begin{proof}
Let $B=\langle e_1, e_2, \ldots \mid S\rangle$ be a countably generated Lie algebra, where $S$ is a GS-basis with respect to the generators $e_i$, for the ordering determined by $e_i < e_j$ if and only if $i < j$. Then $B$ is embeddable in the Lie algebra $P$ with generators $e_1<e_2< \cdots < e_n < \cdots < b< a$ and defining relators:
\begin{enumerate}
 \item[(i)] $s$, for $s \in S$,
 \item[(ii)] $a^2 b^i a b$, for $1 \leq i \leq q+m$, and
 \item[(iii)] $a^2 b^{q+m+i} a b - e_i$, for $i \geq 1$.
\end{enumerate}
There is overlap only between elements of $S$, but their compositions are solvable because $S$ is already a GS-basis. So the set of defining relators of $P$ is a GS-basis; this explains why the obvious map $B \to P$ is injective. Finally, it is clear that $P$ is a quotient of $L$.
\end{proof}

\begin{prop}
 $L$ has trivial center.
\end{prop}

\begin{proof}
Let $f \in L \smallsetminus K \cdot x_1$ be any element and let us check that $f$ does not commute with $x_1$. Write $f$ as a linear combination of regular words on the generators $x_1, \ldots, x_d, b, a$. Since $x_1$ is the minimum letter in the lexicographic order, we have $\bar{f} >_{lex} x_1$, so the leading word of $[f,x_1]$ is $\bar{f} x_1$ by Lemma~\ref{LT.product}. Since the leading words of elements in $G$ do not involve $x_1$, we see that $\bar{f} x_1$ is $G$-irreducible as soon as $\bar{f}$ is $G$-irreducible. By Theorem~\ref{thmShirshov} we see that $[f,x_1]\neq 0$. Thus the centralizer of $x_1$ in $L$ is $K \cdot x_1$, and consequently $Z(L) = 0$.
\end{proof}

In order to conclude the proof of Theorem~\ref{RipsVersion}, we need to show that $L$ is residually nilpotent. It will be enough to prove the following claim.

\bigskip
\textbf{Claim:} If $0 \neq f = \sum \lambda_i f_i$ represents an element of $L^n$, where each $f_i$ is a $G$-irreducible regular word and $\lambda_i \in K$, then $n$ is bounded from above in terms of $\len(\bar{f})$.

\bigskip

 Here $L^n$ is the $n$-th term of the lower central series of $L$. If the claim is true, then no element $f \in L \smallsetminus\{0\}$ can be in $\cap_{n =1}^{\infty} L^n$, that is, $L$ is residually nilpotent.

\bigskip
\emph{Proof of the Claim:} Let $f$ be as above. By hypothesis $f$ is equivalent modulo $G$ to a linear combination $g = \sum_i \mu_i g_i$ of non-associative regular words of length at least $n$.

Recall from Section~\ref{prelim.GS} the process of obtaining $f$ from $g$: if some $\bar{g}_i$ contains some $\bar{r}$ as a subword, with $r \in G$, we change $g$ with $g' = g- \mu_i p$, where $p \in \llangle r \rrangle$ and $\bar{p} = \bar{g}_i$. Thus the summand associated with $\bar{g}_i$ is replaced with a linear combination of deg-lex smaller terms involving some letters of $\bar{g}_i$ and some letters of the words in the expression of $r$. After repeating this process finitely many times, we have rewritten $g$ as a linear combination of $G$-irreducible words.

It follows that there is a sequence of words
\[ w_0 \to w_1 \to w_2 \to \cdots \to w_t = \bar{f},\]
where $w_0$ is one of the $\bar{g}_i$, and either $w_{j+1}$ and $w_j$ have the same length, or $w_{j+1}$ is obtained from $w_j$ by changing a subword of the form $a^2b^kab$ with some word involved in one of the relators $r \in G$, \emph{but not the leading one}, and possibly rearranging the order of the letters. In particular, the resulting word has its length decreased by at most some fixed amount $C>0$, while at the same time increasing the number of instances of one of the letters $x_i$. This is because all elements of $G$ are the sum of a highest length term involving only $a$ and $b$, with some other shorter words all of which involve some $x_i$. The constant $C$ depends only on the length of the words involved in elements of $G$.

Thus $n \leq (C+1) \cdot \len(\bar{f})$. Indeed, in the worst case the word $w_t$ has precisely $t$ letters from $\{x_1, \ldots, x_d\}$, and $w_0$ is a word of length $(C+1) \cdot t$ that does not involve any $x_i$. This concludes the proof of the claim, and thus of Theorem~\ref{RipsVersion} as well.

\section{Applications}  \label{Applications}
In this section we prove Corollary~\ref{coro.paraequiv}. In what follows, we denote by \[H_i(L)=H_i(L,K)= Tor_i^{U(L)}(K,K)\] the $i$-th homology group of a Lie algebra $L$ over $K$. We will make use of the Lyndon-Hochschild-Serre (LHS) spectral sequence associated with a short exact sequence of Lie algebras (see e.g. \cite[Section~7.5]{Weibel}).

First, we need a
tool to determine that certain homomorphisms of Lie algebras induce isomorphisms on the lower central quotients.

\begin{lemma}  \label{LemmaStallings}
Suppose that $\varphi \colon A \to B$ is a homomorphism of Lie algebras that induces an isomorphism on $H_1(-)$ and a surjection on $H_2(-)$. Then $\varphi$ induces isomorphisms $A/A^n \to B/B^n$ for all $n \geq 1$.
\end{lemma}
\begin{proof}
See \cite[Theorem.~2.14]{IMZ2020} or \cite{KnusStammbach1967}.
\end{proof}

Choose a finitely presentable Lie algebra $Q$ of finite cohomological dimension $\cd Q \geq 2$ and satisfying \[H_1(Q)=H_2(Q)=0.\] For instance, $Q= \mf{sl}_2(K)$ will do if $\charac K \neq 2$. Form the exact sequence
 \begin{equation}   \label{exten}
0 \to N \to L \to Q \to 0
 \end{equation}
with Theorem \ref{RipsVersion}. By construction $L$ is residually nilpotent, and therefore so is $N$. Moreover, $L$ is finitely presented of cohomological dimension at most $2$ and $N$ is finitely generated. If $N$ were finitely presented, then by a version of Feldman's theorem (see \cite[page~70]{BieriBook}) for Lie algebras, the cohomological dimension of $L$ would be equal to $\cd(N) + \cd(Q) > 2$, which is a contradiction.

\begin{prop} \label{induces.iso}
 The inclusion $N \subseteq L$ in \eqref{exten} induces isomorphisms $N/N^n \simeq L/L^n$ for all $n \geq 1$.
\end{prop}
\begin{proof}
Using that $H_1(Q)=0$, one verifies directly from the presentation of $L$ that $N \hookrightarrow L$ induces an isomorphism $H_1(N) \to H_1(L)$. Next, look at the LHS spectral sequence $E_{p,q}^2 = H_p(Q,H_q(N))$ converging to $H_2(L)$. By choice $E_{2,0}^2=H_2(Q) = 0$. By construction $Q$ acts trivially on $H_1(N)$, so $E_{1,1}^2=H_1(Q,H_1(N)) \simeq H_1(Q) \otimes H_1(N) = 0$. It follows that $H_2(L)$ is isomorphic to $E_{0,2}^{\infty}$.

We have $E_{0,2}^2 = H_0(Q,H_2(N)) = H_2(N)_Q$. Using again that $Q$ acts trivially on $H_1(N)$, one finds that $E_{2,1}^2  = 0$, and so $E_{0,2}^2= E_{0,2}^3$. Moreover, $E_{0,2}^{4}  = E_{0,2}^\infty$ is the cokernel of the differential $d \colon E_{3,0}^3 = H_3(Q) \to E_{2,0}^3 = H_2(N)_Q$. It follows that the induced map $H_2(N) \to H_2(L)$ is the composite
\[ H_2(N) \twoheadrightarrow H_2(N)_Q \twoheadrightarrow E_{0,2}^{\infty} \simeq H_2(L),\]
which is clearly surjective. By Lemma~\ref{LemmaStallings}, the result follows.
\end{proof}

By taking $A=N$ and $B=L$, we obtain Corollary~\ref{coro.paraequiv}.

\bigskip

Next, we wish to find a pair $(A,B)$ such as the one above, but with the extra property that $H_2(A)$ is infinite-dimensional. In order to do so, it is enough to follow the construction  with a Lie algebra $Q$ with the additional property that $H_3(Q)$ is of infinite dimension. Indeed, arguing as in the proof of Proposition~\ref{induces.iso}, one finds that $E_{3,0}^{\infty} = E_{3,0}^{4}$ is the kernel of that same differential $d \colon H_3(Q) \to H_2(N)_Q$. As the spectral sequence converges to $H_3(L)=0$ such kernel should be trivial, so if $H_3(Q)$ is of infinite dimension, then so is $H_2(N)$.

It remains to give an example of a Lie algebra $Q$ with satisfying the required properties. First, let $V$ be the Lie algebra from Example~\ref{exV}. Recall that $V$ is perfect, finitely presented, of cohomological dimension $2$, and contains a non-abelian free subalgebra $F= \langle y, z \rangle$. Next, consider the subalgebra $S \subseteq F \oplus F \subseteq V \oplus V$ generated by the elements of the form $(f,-f)$, for all $f \in F$. It is finitely generated, but $H_2(S)$ is infinite-dimensional: see e.g. \cite[Lemma 5.5]{WCCLA2019}. Now, let $W = (V \oplus V) \ast_S (V \oplus V)$ be the amalgamated product. It is clear that $W$ is finitely presented and perfect, and moreover a piece of the associated Mayer-Vietoris sequence gives:
 \[ \cdots \to H_3(V \oplus V)^2 \to H_3(W) \to H_2(S) \to H_2(V \oplus V)^2  \to \cdots\]
 Since $H_3(V \oplus V)^2 = 0 = H_2(V \oplus V)^2$, we find that $H_3(W)$ is infinite-dimensional. Finally, we pass to the universal central extension $Q$ of $W$ (see e.g. \cite[Section~7.9]{Weibel}), which is still perfect, finitely presented, of finite cohomological dimension with $H_3(Q)$ infinite-dimensional, but has the extra property that $H_2(Q)=0$. This ends the proof of Corollary~\ref{coro.paraequiv}.

\section{Embeddings into cohopfian Lie algebras}  \label{sec.embeddings}
In this section we prove the first part of Theorem~\ref{thm.embedding}.

Let $H$ be a countably generated Lie algebra. Suppose that $A$ and $B$ are Lie algebras of the same dimension $n \geq 4$ that do not embed in $H$, and such that $F(A) \neq 0$ and $\dim F(B)\geq 2$. This means that we can choose a non-trivial element $a \in A$ (resp. a pair of linearly independent elements $b,b' \in B$) that are non-generators, that is, they can be omitted from any generating set of $A$ (resp. of $B$).

Choose an ordered basis $\{h_1 <  h_2 < \cdots \}$ of $H$ and ordered bases $\{a_1 <  \cdots < a_n\}$ and $\{b_1 < \cdots < b_n\}$ of $A$ and $B$, respectively, with the property that the elements $a_1$, $b_1$ and $b_n$ are non-generators. Write the structure constants of $A$, $B$ and $H$ as
\[ [a_i, a_j] = \sum_{t} \alpha_{i,j}^t a_t, \  \  [b_i, b_j] = \sum_{t} \beta_{i,j}^t b_t, \ \ \text{and } \ \ [h_i, h_j] = \sum_{t} \theta_{i,j}^t h_t,\]
for $i>j$.

Let $X$ be the set of all the $a_i$, $b_i$ and $h_i$, with the ordering extending the orderings above and satisfying $h_i < b_j < a_k$, for all possible $i$, $j$ and $k$. Let $L$ be the Lie algebra generated by $X$ with defining relators:
\begin{enumerate}
 \item $[a_i, a_j] - \sum_{t} \alpha_{i,j}^t a_t$ for $1 \leq j < i \leq n$,
 \item $[b_i, b_j] - \sum_{t} \beta_{i,j}^t b_t$ for $1 \leq j < i \leq n$,
 \item $[h_i, h_j] - \sum_{t} \theta_{i,j}^t h_t$ for $1\leq j<i$, and
 \item $a_n^2 b_1^i a_n b_1 - h_i$ for $i \geq 1$.
\end{enumerate}
Write $G$ for the set of all these relators. Observe that this is a Gröbner-Shirshov basis. Indeed, the only possible compositions are between two relators of the same type (1), (2) or (3), and these are solvable by choice. It follows that the map $H \to L$ that takes $h_i \in H$ to its twin $h_i \in L$ is well defined and injective, since the one-letter words $h_i$ are all $G$-irreducible. So $H$ embeds into $L$.

We wish to show that $L$ is cohopfian. For that, we will show that an injective homomorphism $\varphi \colon L \to L$ takes the subspace $A+B$ onto itself, therefore is surjective too.

In order to see what $\varphi(A)$ and $\varphi(B)$ can be, we study Lie subalgebras of $L$ of dimension $n$. Suppose that $\g \subset L$ is one of these. We may choose a basis $g_1, \ldots, g_n$ of $\g$, all of its elements written as linear combinations of $G$-irreducible regular words, such that
\[ \bar{g}_1 <_{lex} \bar{g}_2 <_{lex} \cdots <_{lex}  \bar{g}_n.\]
The words $\bar{g}_i  \bar{g}_j$ must be $G$-reducible for all $i>j$. Indeed, since $\langle g_i, g_j\rangle$ is finite-dimensional, the set  $\{\delta_k \coloneqq [\underbrace{ g_i, \ldots, g_i}_{k \text{ times}},g_j] \mid k \geq 1\}$ is linearly dependent. In particular, some $\bar{\delta}_k = \bar{g}_i^k \bar{g}_j$ is $G$-reducible, and by inspection of the leading words of $G$ we see that  $\bar{g}_i  \bar{g}_j$ must be $G$-reducible already.

Suppose that $\bar{g}_i \bar{g}_j = u a_n^2 b_1^k a_n b_1 v$ for some $i>j$ and some $k$. Then $\bar{g}_i$ ends with $a_n^2 b_1^k$ and $\bar{g}_j$ starts with $a_n b_1$. Since $a_n$ is the highest letter in our alphabet, $\bar{g}_i$ must start with $a_n^2$. By inspection of the highest words of elements of $G$, this is possible only if $i=n$, because otherwise $\bar{g}_n \bar{g}_i$ is $G$-irreducible. Now look at $\bar{g}_{n-1}$ and $\bar{g}_{n-2}$: both must start with $a_n b_1$, so that $\bar{g}_n \bar{g}_{n-1}$ and $\bar{g}_n \bar{g}_{n-2}$ can be reducible, but at the same time $\bar{g}_{n-1}$ must end with $a_n^2 b_1^r$ for some $r$, so that
$\bar{g}_{n-1} \bar{g}_{n-2}$ can be reducible. This is a contradiction, as a word of the form $a_n b_1 u a_n^2 b_1^r$ cannot be regular.

We conclude that all $\bar{g}_i \bar{g}_j$ contain as a subword the leading word of some relator of type (1), (2) or (3). That is, for all $i>j$, by putting together the last letter of $\bar{g}_i$ and the first letter of $\bar{g}_j$, we find $h_k h_l$, $b_k b_l$ or $a_k a_l$ for some $k>l$.

\subsection*{Case 1:} $\bar{g}_n \bar{g}_1 = u h_k h_l v$ for some $k>l$.

\bigskip

We see that $\bar{g}_n$ ends with $h_k$, and so any $\bar{g}_j$, for $j \geq 2$, begins with some $h_i$. It follows that in all $\bar{g}_i \bar{g}_j$ we find a subword of type $h_k h_l$.
Regular $G$-irreducible words that start with some $h_i$ have length $1$, so the sequence of the $\bar{g}_i$ is of the form
\begin{equation} \label{seqcase1}
 h_{j_1} <_{lex} \cdots <_{lex} h_{j_{n-1}} <_{lex}  u_1 h_{j_n},
\end{equation}
with $j_1 < \cdots < j_{n-1} < j_n$, for some word $u_1$. We need to show that $u_1$ is trivial.

\begin{lemma} \label{L1}
 Suppose that $u = u_0 h_{i_1} \cdots h_{i_n}$ is regular and $G$-irreducible, where $u_0$ is not trivial, its last letter is not some $h_i$, and $i_1 \leq i_2 \leq \cdots \leq i_n$. Then for every $j<i_n$, the highest $G$-irreducible word of $[u,h_j]$ is
 $u_0 h_{i_1} \cdots h_{i_t} h_j h_{i_{t+1}}\cdots h_{i_n}$, where $i_{t} \leq j \leq i_{t+1}$.
\end{lemma}

\begin{proof}
By induction on the length of $u_0$. For the base case $u_0 = c$ is some $a_i$ or $b_i$ and we have:
\[ [ u,h_j] = [ [c h_{i_1} \cdots h_{i_{n-1}}, h_{i_n}] , h_j]
= [[c h_{i_1} \cdots h_{i_{n-1}}, h_j], h_{i_n}] + [c h_{i_1} \cdots h_{i_{n-1}}, [h_{i_n}, h_j]].\]
We may assume by induction (now on $n$) that the highest $G$-irreducible term of $[c h_{i_1} \cdots h_{i_{n-1}}, h_j]$ is of the form $c h_{i_1} \cdots h_j \cdots h_{i_{n-1}}$, and then $[[c h_{i_1} \cdots h_{i_{n-1}}, h_j], h_{i_n}]$ has highest irreducible word $c h_{i_1} \cdots h_j \cdots h_{i_{n-1}}h_{i_n}$ by Lemma~\ref{LT.product}. For the second summand, we substitute $[h_{i_n}, h_j]$ with $\sum_t \theta_{i_n,j}^t h_t$, which results in a term of lower degree.

Now assume that $u_0$ has length at least $2$ and write $u =u_1 u_2 h_{i_1} \cdots h_{i_n}$ where the regular bracketing of $u$ is $[ u_1, u_2 h_{i_1} \cdots h_{i_n}]$. Then
 \[ [u, h_j] = [ [u_1,h_j], u_2 h_{i_1} \cdots h_{i_n}] + [ u_1, [u_2 h_{i_1} \cdots h_{i_n}, h_j ]] \eqqcolon  V+W.\]
 Clearly $u_1 >_{lex} h_j$, so by Lemma~\ref{LT.product} the highest $G$-irreducible term of    $[u_1, h_j]$ is smaller or equal to $u_1 h_j$. The inequality will be strict when $u_1$ ends with some $h_k$, for $k> j$, in which case $u_1 h_j$ is $G$-reducible. Thus, depending on which of the monomials is lexicographically larger, and using Lemma~\ref{LT.product} again, the highest $G$-irreducible term of $V$ is at most the maximum between
 \begin{equation} \label{LT.possibilities}
  u_1 h_j  u_2 h_{i_1} \cdots h_{i_n} \text{ and }  u_2 h_{i_1} \cdots h_{i_n} u_1 h_j.
 \end{equation}

Next, consider $W$. By induction, we may assume that the highest $G$-irreducible term of the bracket $[u_2 h_{i_1} \cdots h_{i_n}, h_j ]$ is $u_2 h_{i_1} \cdots h_{i_t} h_j h_{i_{t+1}}\cdots h_{i_n}$. Since $u_1 >_{lex} u_2 h_{i_1} \cdots h_j \cdots h_{i_n}$ (otherwise $u$ is not regular), the highest $G$-irreducible term of $W$ is precisely $\bar{W} = u_1 u_2 h_{i_1} \cdots h_j \cdots h_{i_n}$, as there is no possible reduction.
This is higher than both possibilities in \eqref{LT.possibilities}, so it is actually the highest $G$-irreducible term of $[u,h_j]$.
\end{proof}

The lemma implies, together with Lemma~\ref{LT.product}, that the brackets of the form $[f,h_j]$ will have leading word of length $\len(\bar{f}) + 1$, except for the cases where $f$ is a linear combination of the $h_i$. So in chain~\eqref{seqcase1},
if $\bar{g}_n = u_1 h_{j_n}$ with $u_1$ non-trivial, then the highest term of $[g_n,g_1]$ is a word of length greater than that of $\bar{g}_n$. In particular, such word is  none of the $\bar{g}_i$, which is a contradiction: $g_n$ and $g_1$ belong to the finite-dimensional algebra spanned by $g_1, \ldots, g_n$, so $[g_n, g_1]$ is a linear combination of such elements. We conclude that $u_1$ must be trivial, that is, $\bar{g}_n=h_{j_n}$.

Finally, since any regular word smaller than some $h_j$ is also of the form $h_k$, we see that each $g_i$ is a linear combination of the $h_j$, so
\emph{$\g$ is a subalgebra of $H$}.

\subsection*{Case 2:} $\bar{g}_n \bar{g}_1 = u b_k b_l v$  for some $k>l$.

\bigskip

By inspecting the relations again, one sees that the sequence of the $\bar{g}_i$ must be of the form
\begin{equation} \label{seqcase2}
b_1 v <_{lex} b_2 <_{lex} \cdots <_{lex} b_{n-1} <_{lex} u b_n,
\end{equation}
for some words $v$ and $u$.

Notice that for $2 \leq i \leq n-1$, the highest term of $g_i$ is $b_i$, a word of length $1$. In particular, $g_i$ must be equal to $b_i$ plus a linear combination of the letters which are lexicographically smaller, that is, $b_j$, with $j <i$, and the $h_j$. We show first that the coefficients of the $h_i$ must be zero.

\begin{lemma}  \label{elim.hj}
In the setting above, $g_i$ is the form $g_i = b_i + \sum_{j<i} x_j b_j$ for all $2 \leq i \leq n-1$, with $x_i \in K$.
\end{lemma}
\begin{proof}
For all $i,j$, the bracket $[g_i,g_j]$ is an element of $\g$, that is, a linear combination of the $g_k$. In particular, its highest term is one of the words in the chain \eqref{seqcase2}. We check what happens when $2\leq j<i \leq n-1$.

To fix ideas, consider the case $j=2$ and $i=3$. We have:
\[ g_2 = b_2 + x_1 b_1 + \sum_{j\geq 1} y_j h_j, \text{  and  } g_3 = b_3 +  r_2b_2 + r_1b_1 + \sum_{j\geq 1} s_j h_j.\]
Taking the bracket $[g_3, g_2]$ and expanding linearly, one sees that all terms of the form $[b_i, b_j]$ or $[h_i,h_j]$ reduce to a term of length one using one of the words of the Gröbner-Shirshov basis. Since $g_2$ does not involve $b_3$, we see that the highest term that remains is of the form $y_j[b_3, h_j]$, for the highest index $j$ such that $s_j \neq 0$. It is clear that the word $b_3 h_j$ is not $G$-reducible, nor is it a word in the chain \eqref{seqcase2}. We must conclude that actually all the $y_j$ are zero.

Knowing that all the $y_j$ are zero, the highest term now becomes $-s_j[b_2, h_j]$, for the highest index $j$ such that $s_j \neq 0$. The same argument shows that all $s_j$ must be zero. So both $b_2$ and $b_3$ must be of the form given in the statement of the lemma, and it is clear that this argument works if we consider any bracket $[g_i,g_j]$ with $2 \leq j < i \leq n-1$.
\end{proof}

We conclude that $\g$ contains the subalgebra of $B$ generated by $g_2, \ldots, g_{n-1}$, which turns out to be the whole algebra $B$, since $b_1$ and $b_n$ are non-generators. So $\g = B$, because they are algebras of the same dimension.

\subsection*{Case 3:} $\bar{g}_n \bar{g}_1 = u a_k a_l v$  for some $k>l$.

\bigskip

The sequence is now
\begin{equation} \label{seqcase3}
a_1 v <_{lex} a_2 <_{lex} \cdots <_{lex} a_{n-1} <_{lex}  a_n,
\end{equation}
for some word $v$. There is no extra word multiplying $a_n$ on the left, as in case $2$, because $a_n$ is the highest letter in our alphabet, so $w=a_n$ is the unique regular word that ends with $a_n$.

In this case, the elements $g_i$, for $2 \leq i \leq n$, are linear combinations of the $a_j$, for $j \leq i$, the $b_j$ and the $h_j$, with leading term $a_i$. The same argument as in Lemma~\ref{elim.hj} shows that the coefficients of the $h_j$ and the $b_j$ are all zero, that is, $g_i$ is of the form $g_i = a_i + \sum_{j<i} x_j a_j$. We use similarly that $a_1$ is a non-generator of $A$ and conclude that $\g = A$.

\bigskip

In summary: any subalgebra $\g \subset L$ of dimension $n$ is a subalgebra of $H$, equal to $A$ or equal to $B$, according to whether it falls into Case 1, 2 or 3, respectively.

\begin{theorem}  \label{thm.cohopfian}
In the setting above, $L$ is cohopfian.
\end{theorem}
\begin{proof}
 Let $\varphi \colon L \to L$ be an injective homomorphism. The image $\varphi(A)$ is a subalgebra of $L$ of dimension $n$. As $A$ is not embeddable in $H$, we conclude that $\varphi(A) = A$ or $\varphi(A) = B$ . Similarly, $\varphi(B) = A$ or $\varphi(B) = A$. By injectivity, one of these images is $A$ and the other is $B$. Since $L$ is generated by $A$ and $B$, we find that $\varphi$ is surjective.
\end{proof}

\section{Embeddings into hopfian Lie algebras} \label{sec.emb.hopfian}
We prove now the second part of Theorem~\ref{thm.embedding}. So assume that $H$ is a countably generated Lie algebra that does not contain isomorphic copies of the nilpotent Lie algebras $A$ and $B$ satisfying the conditions of  the theorem.

Let $S$ be the non-abelian Lie algebra of dimension $2$. First of all, we will replace $H$ with $\widetilde{H} = (S \oplus K) \ast H$. This new algebra is again countably generated and does not contain isomorphic copies of $A$ and $B$ (by e.g. \cite[Corollary~3.1]{Feldman1983}). The effect of such replacement is that we can now easily choose an ordered basis $h_1, h_2, \ldots$ of $\widetilde{H}$ with the property that
 \[[h_1, h_2] = h_1 = [h_1, h_3].\]
Following the construction in Section~\ref{sec.embeddings}, we find a quotient $L$ of $A \ast B$ that is cohopfian and contains an isomorphic copy of $\widetilde{H}$, and thus of $H$ too.

 Let $\psi \colon L \to L$ be a surjective homomorphism. We wish to show that $\psi$ is an isomorphism. Let $m$ be an integer greater than the nilpotency class of both $A$ and $B$. The induced homomorphism $\bar{\psi} \colon L/L^m \to L/L^m$ is again surjective, so it is actually an isomorphism, since $L/L^m$ is finite-dimensional. Since $A$ and $B$ intersect $L^m$ trivially, as they survive in the nilpotent quotients $L/\llangle B \rrangle \simeq A$ and $L/\llangle A \rrangle \simeq B$,
 we find that $\psi |_A$ and $\psi|_B$ must be already injective. So both $\psi(A)$ and $\psi(B)$ are subalgebras of $L$ of dimension $n$, and we conclude as before that $\psi(A)=A$ and $\psi(B)=B$.

 The defining relations of $L$ impose strong constraints on the images of the generators $a_n$ and $b_1$ under homomorphisms.

 \begin{lemma}
  In the setting above, $\psi(a_n) \in K\cdot a_n$ and $\psi(b_1) \in K \cdot b_1$.
 \end{lemma}
\begin{proof}
  Write $\psi(a_n) = \sum_j \lambda_j a_j$ and $\psi(b_1) = \sum_j \mu_j b_j$, and assume that $\overline{\psi(a_n)} = a_r$ and $\overline{\psi(b_1)} = b_s$. In $L$ we have \[h_1 = a_n^2 b_1 a_n b_1 = [[a_n, [a_n, b_1]], [a_n, b_1]].\]
 Let $P$ be the result of applying $\psi$ to the generators in the bracket above, that is:
 \[P \coloneqq [[\sum_j \lambda_j a_j, [\sum_j \lambda_j a_j, \sum_j \mu_j b_j]], [\sum_j \lambda_j a_j, \sum_j \mu_j b_j]]\]
 Then $P$ is equivalent to $\psi(h_1)$ in $L$. It is not hard to see that its highest term is $(\lambda_r^3 \mu_s^2) \cdot a_r^2 b_s a_r b_s$. Similarly, by applying $\psi$ to
 \[h_2 = a_n^2 b_1^2 a_n b_1 = [[a_n, [[a_n, b_1],b_1]], [a_n, b_1]],\]
 we find the element
 \[Q \coloneqq [[\sum_j \lambda_j a_j, [[\sum_j \lambda_j a_j, \sum_j \mu_j b_j],\sum_j \mu_j b_j]], [\sum_j \lambda_j a_j, \sum_j \mu_j b_j]],\]
which is equivalent to $\psi(h_2)$ and has highest term $(\lambda_r^3 \mu_s^3) \cdot a_r^2 b_s^2 a_r b_s$.

 As $f \coloneqq [h_1,h_2]-h_1$ is zero in $L$, the leading word of $[P,Q]-P$ must be $G$-reducible. Since
 $a_r^2 b_s a_r b_s >_{lex} a_r^2 b_s^2 a_r b_s$, we see that such word is
 $a_r^2 b_s a_r b_s a_r^2 b_s^2 a_r b_s$ by Lemma~\ref{LT.product}. This can only be reducible when $r = n$ and $s=1$. From that we deduce that $\psi(b_1) = \mu_1 b_1$ and that $\psi(a_n)$ has leading term $\lambda_n a_n$.

 We have proved that the leading word of $P$ is $a_n^2 b_1 a_n b_1$, which is equivalent to $h_1$ in $L$. Now suppose for contradiction that $\psi(a_n)$ is not a multiple of $a_n$ and let $a_m$ be the second highest term appearing with non-zero coefficient in the sum $\psi(a_n) = \sum_j \lambda_j a_j$.
 By replacing the leading term of $P$ with a multiple of $h_1$, we find a new element $\widetilde{P}$, equivalent to $P$ in $L$, with highest word $a_n^2 b_1 a_m b_1$ (the second highest appearing in $P$). Similarly, by replacing the highest term of $Q$ with a multiple of $h_2$, we find an equivalent element $\widetilde{Q}$ with highest word $a_n^2 b_1^2 a_m b_1$.

Now, since $a_n^2 b_1 a_m b_1 >_{lex} a_n^2 b_1^2 a_m b_1$, we have again by Lemma~\ref{LT.product} that $[\widetilde{P},\widetilde{Q}]-\widetilde{P}$ is equivalent in $L$ to an element with leading word $a_n^2 b_1 a_m b_1 a_n^2 b_1^2 a_m b_1$, which is $G$-irreducible. This is a contradiction, so $\psi(a_n) \in K \cdot a_n$.
\end{proof}

 We may now write $\psi(a_n) = \lambda a_n$ and $\psi(b_1) = \mu b_1$. Thus the action of $\psi$ on $\widetilde{H}$ is given by
 \[\psi(h_j) = \psi(a_n^2 b_1^j a_n b_1) = \lambda^3 \mu^{j+1} h_j\]
 for all $j \geq 1$. Since $[h_1, h_2] = [h_1, h_3] = h_1 \neq 0$, we find that $\lambda^6 \mu^5 = \lambda^6 \mu^6$, thus $\mu = 1$, and also $\lambda^6\mu^6 = \lambda^3 \mu^2$, so $\lambda^3 =1$. It follows that $\psi|_{\widetilde{H}} = \lambda^3 \cdot \mathrm{id}_{\widetilde{H}} = \mathrm{id}_{\widetilde{H}}$.

 Finally, we can see from the relations of $L$ that there is a well-defined homomorphism $\rho \colon L \to L$ determined by
 \[ \rho|_A = (\psi |_A)^{-1}, \  \ \rho|_B = (\psi |_B)^{-1} \  \ \text{ and } \  \ \rho|_{\widetilde{H}} = \mathrm{id},\]
 which clearly is an inverse for $\psi$. This completes the proof of Theorem~\ref{thm.embedding}.

\section*{Acknowledgements}
The author is grateful to F.Y. Yasumura for helpful discussions about the paper.

\end{document}